\documentclass[12pt,reqno]{amsart}

\usepackage{amsfonts,amsmath,amsthm}
\usepackage{amssymb,epsfig}
\usepackage{cases}
\usepackage[usenames]{color}
\usepackage{enumerate} 

\usepackage{color} 
\definecolor{vert}{rgb}{0,0.6,0}

\theoremstyle{plain}
\newtheorem{thm}{Theorem}[section]

\newtheorem{lem}[thm]{Lemma}

\newtheorem{prop}[thm]{Proposition}
\theoremstyle{remark}
\newtheorem{rem}{\bf{Remark}}
\numberwithin{equation}{section}

\newcommand{\E}{\mathbb{E}}
\newcommand{\M}{\mathbb{M}}
\newcommand{\N}{\mathbb{N}}

\newcommand{\R}{\mathbb{R}}

\newcommand{\T}{\mathbb{T}}
\newcommand{\Z}{\mathbb{Z}}

\newcommand{\cL}{\mathcal{L}}

\newcommand{\Li}{L^{\infty}}
\newcommand{\Lip}{{\rm Lip\,}}

\newcommand{\Q}{\mathbb{T}^{n}\times(0,\infty)}

\newcommand{\cQ}{\mathbb{T}^{n}\times[0,\infty)}

\newcommand{\al}{\alpha}
\newcommand{\gam}{\gamma}
\newcommand{\del}{\delta}
\newcommand{\ep}{\varepsilon}

\newcommand{\lam}{\lambda}
\newcommand{\sig}{\sigma}

\newcommand{\Del}{\Delta}

\newcommand{\pl}{\partial}

\newcommand{\tr}{{\rm tr}\,}

\begin{document}

\title[Large-time behavior for obstacle problems]
{Large-time behavior for obstacle problems\\ 
for degenerate viscous Hamilton--Jacobi equations}

\author[H. MITAKE, H. V. TRAN]
{Hiroyoshi Mitake and Hung V. Tran}

\thanks{
The work of HM was partially supported by KAKENHI  
\#24840042, JSPS and Grant for Basic Science Research Projects from the 
Sumitomo Foundation. 
}

\address[H. Mitake]
{
Department of Applied Mathematics, 
Faculty of Science, Fukuoka University, 
Fukuoka 814-0180, Japan}
\email{mitake@math.sci.fukuoka-u.ac.jp}

\address[H. V. Tran]
{
Department of Mathematics, 
The University of Chicago, 5734 S. University Avenue, Chicago, Illinois 60637, USA}
\email{hung@math.uchicago.edu}


\maketitle

\begin{abstract}
Cagnetti, Gomes, Mitake and Tran (2013) introduced a new idea to study the large time behavior for degenerate viscous Hamilton--Jacobi equations. In this paper, we apply the method to study the large-time behavior of the solution to the obstacle problem for degenerate viscous Hamilton--Jacobi equations.  
We establish the convergence result under rather general assumptions. 
\end{abstract}


\tableofcontents


\section{Introduction and Main Result}
In this paper we study the large time behavior of the solution to the obstacle problem for degenerate viscous Hamilton--Jacobi equations: 
\begin{equation} \notag
{\rm (C)} \qquad 
\begin{cases}
\max\{u_t -\tr\big(A(x)D^2u\big) + H(x,Du), u-\psi\}=0 & \text{ in } \Q, \\
u(x,0)=u_{0}(x) & \text{ on } \T^n,
\end{cases}
\end{equation}
where $\T^n$ is the $n$-dimensional torus $\R^n / \Z^n$.
Here $u_t, Du, D^2u$ denote the partial derivative with respect to $t$, 
the spatial gradient and Hessian of the unknown
$u:\Q \to \R$, respectively.  
The functions $H:\T^n \times \R^n \to \R$, 
$A:\T^n \to \M^{n\times n}_{\text{sym}}$, and $\psi:\T^n\to\R$ 
are the given Hamiltonian, diffusion matrix, and obstacle function,  respectively, where $\M^{n\times n}_{\text{sym}}$ is the set of $n\times n$
real symmetric matrices. 
We assume the following conditions (H1)--(H4) throughout this paper:  
There exist some constants $C,\theta>0$ so that
\begin{itemize}
\item[(H1)] $H \in C^2(\T^n \times \R^n)$, and 
$D^2_{pp}H \ge 2\theta I_n$, where $I_n$ is the identity matrix of size $n$, 
\item[(H2)] 
$|D_x H(x,p)| \le C(1+|p|^2)$,  
\item[{(H3)}] 
$A(x)=(a^{ij}(x))\in \M^{n\times n}_{\text{sym}}$ with $A(x)\ge 0$,  and $A\in C^2(\T^n)$, 
\item[{(H4)}] 
$\psi\in C^2(\T^n)$ and $u_0 \in \Lip(\T^n)$ satisfying the compatibility condition to the initial data, i.e., $u_0 \leq \psi$ on $\T^n$.
\end{itemize}
\smallskip

Obstacle problem is an important class of problems in analysis and applied mathematics,
which has been developed first in the study of PDEs, variational inequalities and free boundary problems. 
We refer to \cite{KS, F} for classical references and 
to \cite{C} for the regularity theory of the free boundary in the context 
of the obstacle problem. 
The obstacle problem (C) was studied first in the context of viscosity solutions 
in the early 1980s  by Lions in his book \cite{L} and 
it is well-known that the obstacle problem (C) has the background 
of the control problem (see \cite{B, BCD} for instance). 
There are also many applications in the study of financial mathematics. 
See \cite{P,FPP,B2,Hy} and references therein for more details. 
We refer to \cite{BFZ} for the application to the reachability and optimal time. 
Note also that problem (C) is equivalent to the followings
\begin{equation} \notag
\begin{cases}
u_t -\tr\big(A(x)D^2u\big) + H(x,Du)\le0, \ \text{and} \ u-\psi\le0 & \text{ in } \Q, \\
\big(u_t -\tr\big(A(x)D^2u\big) + H(x,Du)\big)\cdot (u-\psi)=0 & \text{ in } \Q, \\
u(x,0)=u_{0}(x) & \text{ on } \T^n,  
\end{cases}
\end{equation}
and hence (C) is also called a \textit{variational inequality of obstacle type} or 
a \textit{complementarity system}. 
\smallskip

Our main concern in this paper is the large-time asymptotics for (C). 
Let us recall quickly the study of the asymptotic behavior of solutions to 
the initial-value problem for possibly degenerate viscous Hamilton--Jacobi equations 
including the first order cases and the non-degenerate viscous cases:  
\begin{equation}\label{eq:HJ}
u_t-\tr\big(A(x)D^2u\big)+H(x,Du)=0, \quad \text{in} \ \Q,  
\end{equation}
which has got much attention and was studied extensively since the late 1990s. 
See \cite{NR, BS, DS,F2,I2008} for Hamilton--Jacobi equations (i.e., when $A\equiv 0$),  
\cite{BS2,IS} for non-degenerate viscous Hamilton--Jacobi equations (i.e. when $A$ is non-degenerate) and \cite{MT3} for  weakly coupled systems. The list of references on this study
here is by no means complete and we refer the readers 
to the aforementioned papers for a complete picture
of the subject.   
Note that though the large time behaviors are of the same type,
the methods in the above references of obtaining convergence results 
are completely different. 
The methods for first order cases are based on a finite speed of propagation, 
a stability of extremal curves in the context of the dynamical approach 
or stability of solutions for time large in the context of the PDE approach, 
while that for the non-degenerate viscous cases are based on the strong comparison principle.
It is clear that these methods are not applicable to the general degenerate viscous case \eqref{eq:HJ} until now because of the lack of both the finite speed of propagation and the strong comparison principle.

Very recently, Cagnetti, Gomes, and the authors in \cite{CGMT}
introduced a new and unified approach to establish large time asymptotics for the general case \eqref{eq:HJ}, where $A$ could be degenerate. 
Under (H1)--(H3), we proved that
\begin{equation}\label{CGMT-g}
u(\cdot,t)+c_H t \to v \quad  \text{uniformly in} \ \T^n \ \text{as} \ t\to\infty ,
\end{equation}
where $(v,c_H)\in C(\T^n)\times \R$ is a solution of the \textit{ergodic problem}: 
\[
{\rm(E)} \qquad
-\tr\big(A(x)D^2v\big)+H(x,Dv)=c_H \quad \text{in} \ \T^n.   
\]
It was established also in \cite{CGMT} that $c_H$ is unique, but $v$ is not unique
 in general even up to some additive constants,
which makes the convergence analysis very delicate. 
We were able to obtain \eqref{CGMT-g} by using the stability of the viscosity solutions 
together with the deep fact that 
\begin{equation}\label{key:stability}
u_t+c_H \to 0 \ \text{as} \ t \to \infty \ \text{in the viscosity sense}, 
\end{equation}
which has been achieved in light of the new conservation of energy identity
and estimates comparing the solutions of the approximated Cauchy problems and ergodic problems via the nonlinear adjoint method introduced by Evans \cite{Ev1}. 
We refer to \cite{Ev1, T1,CGT1,CGT2} for the details of the nonlinear adjoint method. 
See also \cite{LN} for another direction of the study of the large time behavior for 
degenerate parabolic equations. 
\smallskip

The analysis of the large time behavior for obstacle problems turns out to be a bit more complicated because of the appearance of the obstacles. 
There are cases where we do not see the obstacles at all 
because the solution $u$ of (C) always stays below $\psi$ for $t$ large enough.
The behavior of $u(\cdot,t)$ as $t\to \infty$ then is basically the same as in the usual case in \cite{CGMT}.
There are however cases where we have to take into account the obstacle $\psi$ and in general
$u(\cdot,t)$ touches $\psi$ at some parts for all time $t$ large enough. We have to provide a different analysis
for this case, which is based on the general method in \cite{CGMT}.

The large time asymptotics for (C) is related to the behavior of two ergodic problems. The first ergodic problem is the usual ergodic problem (E) without the obstacle,
which gives the characterization to the two phenomena described above.

\begin{prop}\label{prop:e.CGMT}
There exists  a solution $(v,c_H)\in C(\T^n)\times\R$ of {\rm(E)}. 
Moreover, $c_H$ is unique. We call $c_H$ the ergodic constant, which is  uniquely determined by $H$ and $A$ by the formula
\begin{equation}\label{c_H}
c_H=\inf\{a\in\mathbb{R}\mid \text{there exists} \ v \ \text{such that} \  
-\tr\big(A(x)D^2v\big)+H(x,Dv)\le a \ \text{in} \ \T^n\}. 
\end{equation}
\end{prop}

More precisely, we observe that, if $c_H>0$, the behavior of $u$ as $t\to\infty$ 
is not affected at all by the obstacle. 
This could be explained heuristically as follows. Pick any $(v,c_H)$ to be a solution of (E) and notice that $v-c_H t$ is a solution of \eqref{eq:HJ}. 
Since $c_H>0$, we obtain the existence of some positive constant $T_0>0$ so that $v-c_H t <\psi$ for all $t \geq T_0$. In particular, $v-c_Ht$ is a solution of (C) for $t\geq T_0$, and the behavior of $v-c_Ht$ as $t\to \infty$ clearly is not involved by $\psi$.
The appearance of the two phenomena where the asymptotics depend on the sign of $c_H$ was first found in the study of the Cauchy-Dirichlet problem in \cite{M2,M3,Ta}.

The most interesting case is when $c_H\leq 0$. One could see that
 the above reasoning does not work, and actually we need to consider the second ergodic problem, which does involve the obstacle $\psi$:
\begin{equation*}
{\rm (EO)}\quad
\max\{-\tr\big(A(x)D^2V\big)+H(x,DV), V-\psi\}=0 \quad \text{in} \ \T^n.
\end{equation*}
Existence result for the above can be described by the sign of $c_H$. In fact, 
if $c_H\le0$, then it is rather easy to prove the existence of solutions by 
the Perron method. 

\begin{prop}\label{prop:EO}
There exists a solution $V\in C(\T^n)$ of {\rm (EO)} if and only if $c_H\le0$. Moreover, if $c_H<0$, then $V$ is unique.
\end{prop}

Based on these existence results on the ergodic problems we establish the main result. 
\begin{thm}[Main Theorem]\label{thm:main}
Let $u$ be the solution of {\rm(C)} and $c_H$ be the constant given by \eqref{c_H}. 
The followings hold:\\
{\rm(i)} 
If $c_H>0$, then 
$u+c_Ht$ converges uniformly to a solution $v$ of {\rm(E)} on $\T^n$.\\
{\rm(ii)} 
If $c_H\leq 0$, then $u$ converges uniformly to a solution $V$ of {\rm (EO)} on $\T^n$.
\end{thm}

Our analysis  to prove Theorem \ref{thm:main} follows the method which was introduced in \cite{CGMT}. 
We rescale (C) in an appropriate way, and introduce an approximation to the rescaled problem by adding a viscosity term and 
a penalized term corresponding to the obstacle. 
By using a long time averaging effect, which have first been found in \cite{CGMT}, 
we prove the key stability result  similar to \eqref{key:stability} and achieve the desired result. 
Notice that the penalized term needs to be handled carefully as it does not appear in the usual setting of \cite{CGMT}.

There are two main difficulties on the penalization in this paper. The first part is on the derivation of some a priori bounds for the penalization and the viscosity solution.
This is a rather classical difficulty, and it could be achieved in light of the maximum principle.
The second part is more subtle as one needs to derive better estimates of the penalization
and its derivative on the support 
of the solution of the adjoint equation, which is the central point of the approach in \cite{CGMT}.
Note that the derivative term is much more dangerous compare to the penalization itself.
We need to choose carefully the different scales of the viscosity term and the penalized term in the approximation 
and obtain some new estimates in this context. See Lemmas \ref{lem:key1} and \ref{lem:key2} for details.

\medskip
The paper is organized as follows: Section 2 devotes to the proofs of  Propositions \ref{prop:e.CGMT}, 
\ref{prop:EO}. 
In Section 3 we introduce the approximation procedure,
derive the new identities and key estimates in Lemmas \ref{lem-1}, \ref{lem:key1} and \ref{lem:key2}, and prove the main theorem, Theorem \ref{thm:main} (ii).


\section{Preliminaries}
We first give the proof of Proposition \ref{prop:e.CGMT}  
for the self-containedness. 
Throughout the paper, we will mostly use Einstein's convention of summation for simplicity except in some cases where the summations are stated clearly in order to avoid any confusion.

\begin{proof}[Proof of Proposition {\rm\ref{prop:e.CGMT}}]
The proof is based on the standard Bernstein method. 
For $\al>0,\del>0$, consider the equation
\[
\al v^{\al,\del}-\tr\big(A(x)D^2v^{\al,\del}\big)+H(x,Dv^{\al,\del})=\del\Del v^{\al,\del} \quad \text{in} \ \T^n. 
\]
Owing to the discount and viscosity terms, there exists a (unique) classical solution $v^{\al,\del}$. 
We easily see that $|\al v^{\al,\del}|\le C$ for $C>0$ by the comparison principle. 
Set $\varphi:=|Dv^{\al,\del}|^2/2$ and then $\varphi$ satisfies 
\[
2\al\varphi-a^{ij}_{x_k}v^{\al,\del}_{x_ix_j}v^{\al,\del}_{x_k}
-a^{ij}(\varphi_{x_ix_j}-v^{\al,\del}_{x_ix_k}v^{\al,\del}_{x_jx_k})
+D_{x}H\cdot Dv^{\al,\del}+D_{p}H\cdot D\varphi=\del(\Del\varphi-|D^2v^{\al,\del}|^2). 
\]

Take a point $x_0$ such that $\varphi(x_0)=\max_{\T^n}\varphi\geq 0$ and note that
at that point
\begin{equation}\label{Bern-0}
-a^{ij}_{x_k}v^{\al,\del}_{x_ix_j}v^{\al,\del}_{x_k}
+D_{x}H\cdot Dv^{\al,\del}+a^{ij}v^{\al,\del}_{x_i x_k}v^{\al,\del}_{x_j x_k}+\del|D^2v^{\al,\del}|^2 \leq 0.
\end{equation}
The two terms $a^{ij}v^{\al,\del}_{x_i x_k}v^{\al,\del}_{x_j x_k}$ and $\del |D^2 v^{\al,\del}|^2$ are the good terms, which will help us to control other terms and to deduce the result.

We first use the trace inequality (see \cite[Lemma 3.2.3]{SV} for instance),
\begin{equation}\label{SV-a2}
(\tr(A_{x_k} S))^2 \leq C \tr(SAS) \quad  
\text{for all} \ S \in \M^{n\times n}_{\text{sym}}, \ 1 \leq k \leq n, 
\end{equation}
for some constant $C$ depending only on $n$ and $\|D^2 A\|_{L^\infty(\T^n)}$
to yield that, for some small constant $c>0$,
\begin{align}
&a^{ij}_{x_k}v^{\al,\del}_{x_ix_j}v^{\al,\del}_{x_k}
=\tr(A_{x_k} D^2v^{\al,\del})v^{\al,\del}_{x_k}
\le c \big(\tr(A_{x_k} D^2v^{\al,\del})\big)^{2}+ \frac{C}{c}|Dv^{\al,\del}|^2\notag\\
\leq&\, 
\frac{1}{2}\tr(D^2v^{\al,\del}A D^2v^{\al,\del})+ C|Dv^{\al,\del}|^2
=\frac{1}{2}a^{ij}v^{\al,\del}_{x_i x_k}v^{\al,\del}_{x_j x_k}+C|Dv^{\al,\del}|^2.\label{Bern-1}
\end{align}

Next, since $A$ is a symmetric positive definite matrix, it  can be diagonalized as $A=P^TDP$ where $D$ is the diagonal matrix, which could be written as $D=\text{diag}\{d^1,\ldots,d^n\}$ with $d^i \geq 0$, 
and $P^TP=I_n$. We have
\begin{align}
&\left(a^{ij} v^{\al,\del}_{x_i x_j} \right)^2=\left(p^{mi} p^{mj} d^m v^{\al,\del}_{x_i x_j}\right)^2\notag\\
\leq\,&
\left(\sum_j C\left|p^{mi}d^mv^{\al,\del}_{x_i x_j}\right| \right)^2
\leq
C \sum_{j,m} \left|\sum_i p^{mi} \sqrt{d^m} v^{\al,\del}_{x_i x_j}\right|^2\notag\\
=\,&
C\sum_{j,m} p^{mi} \sqrt{d^m}v^{\al,\del}_{x_i x_j} p^{mk}\sqrt{d^m}v^{\al,\del}_{x_kx_j}
=
Cp^{mi}p^{mk}d^mv^{\al,\del}_{x_i x_j} v^{\al,\del}_{x_k x_j}
=
C a^{ik} v^{\al,\del}_{x_i x_j} v^{\al,\del}_{x_k x_j}. \label{Bern-2}
\end{align}
In light of \eqref{Bern-2}, for some $c>0$ sufficiently small,
\begin{align}
\frac{1}{2}a^{ij}v^{\al,\del}_{x_i x_k}v^{\al,\del}_{x_j x_k}+\del|D^2 v^{\al,\del}|^2
&\geq 
4c\left(\left(a^{ij}v^{\al,\del}_{x_ix_j}\right)^2+(\del \Del v^{\al,\del})^2 \right)\notag\\
&\geq 
2c \left(a^{ij}v^{\al,\del}_{x_i x_j}+\del \Del v^{\al,\del}\right)^2
=
2c \left (\al v^{\al,\del}+H(x,Dv^{\al,\del})\right)^2 \notag\\
&\geq
c H(x,Dv^{\al,\del})^2-C.  \label{Bern-3}
\end{align}
Combining \eqref{Bern-0}, \eqref{Bern-1}, and \eqref{Bern-3} to achieve that
$$
D_xH\cdot Dv^{\al,\del}-C|Dv^{\al,\del}|^2+c H(x,Dv^{\al,\del})^2 \leq C.
$$
We then use (H1) and (H2) in the above to get the existence of a constant $C>0$ independent of $\al,\del$ so that $|Dv^{\al,\del}(x_0)| \leq C$.
Therefore, setting $w^{\al,\del}(x):=v^{\al,\del}(x)-v^{\al,\del}(0)$, 
by passing some subsequences if necessary, we can send $\del$ and $\al$ to $0$ in this order to yield that $w^{\al,\del}$ and $\al v^{\al,\del}$, respectively, uniformly converge $v$ and $-c_H$ which satisfies (E) in the viscosity sense. 

We write $c$ for the right hand side of \eqref{c_H}. By the definition, we have $c\le c_H$. 
Suppose that $c<c_H$. We may choose a subsolution $(w,c)$ and a solution $(v,c_H)$ of (E), respectively, such that $w>v\ge0$ on $\T^n$. 
Then there exists a small $\al>0$ such that 
\[
\al w-\tr\big(A(x)D^2w\big)+H(x,Dw)\le 
\al v-\tr\big(A(x)D^2v\big)+H(x,Dw) \quad \text{in} \ \T^n.  
\]
By the comparison principle, $w\le v$ on $\T^n$, which gives a contradiction. 
The uniqueness of the ergodic constant is also obtained by the same argument. 
\end{proof}

We provide a priori bound for the solution of (C) 
by using Proposition \ref{prop:e.CGMT}. 
\begin{prop}\label{prop:bound}
Let $v\in C(\T^n)$ be a solution of {\rm (E)}.
There exists a universal constant $C>0$ so that the following boundedness results hold:
\\
{\rm(i)} 
If $c_H>0$, then 
$$
v-C \leq u+c_Ht \leq v+ C \qquad \text{in}\ \cQ.
$$
{\rm(ii)} 
If $c_H\leq 0$, then
$$
v-C \leq u \leq \psi \qquad \text{in}\ \cQ.
$$
\end{prop}

\begin{proof}
We  choose $C>0$ sufficiently large so that $|u_0-v|\le C$.
Note that $v+C-c_Ht$ is a supersolution of (C), 
$v-C-c_Ht$ is a solution of \eqref{eq:HJ}, and 
\[
v-C-c_Ht \leq v-C\leq u_0 \leq \psi \quad \text{on} \ \cQ.
\]
Thus, by the comparison principle for (C), we get the result of (i).

If $c_H \leq 0$ then $v-C$ is a subsolution of (C).
Thus, in light of the comparison principle, (ii) holds.
\end{proof}

As stated in Introduction heuristically,  
the solution of (C) is not influenced by the obstacle $\psi$ for a large time, if 
$c_H > 0$. 
Indeed, 
we take $T_0>0$ sufficiently large so that, in light of (i) of Proposition \ref{prop:bound},
$$
u(\cdot,t) < \psi \quad \text{on} \ \T^n \ \text{for all} \ t \geq T_0.
$$
In particular, $u$ does not touch the obstacle $\psi$ for all $t \geq T_0$, and it solves the usual Hamilton--Jacobi equation
$$
u_t-\tr\big(A(x)D^2u\big)+H(x,Du)=0 \quad \text{in} \ \T^n \times [T_0,\infty).
$$
We can thus apply the large time behavior result \cite[Theorem 1.1]{CGMT} to deduce the result of Theorem \ref{thm:main} (i). 

\begin{rem}
We can observe this phenomenon from the control viewpoint. 
We consider the following stochastic optimal control problem: 
\begin{align}
&\text{Minimize}&  &
\E_{x}\Big[\int_{0}^{\theta}
L(X^{\xi}(s),\xi(s))\,ds
+h(X^{\xi}(\theta),t-\theta)
\Big], 
\nonumber\\
&\text{subject to} & &
dX^{\xi}(s)=-\xi(s)\,ds+\sig(X^{\xi}(s))dW(s), \
X^{\xi}(0)=x,  \nonumber
\end{align}
over all controls $\xi\in L^{1}([0,t])$ and $\theta\ge0$ 
for all $(x,t)\in\cQ$, where 
\[h(x,s):=\psi(x) \ \text{for} \ s>0, \ \text{and} \ 
h(x,0):=u_{0}(x).  
\]
The matrix  $\sig(x)\in\M^{m\times n}$ is the square root of $A(x)$, i.e., $A(x)=\sig(x)\sig^{T}(x)$, 
$W$ is the $n$-dimensional Brownian motion, 
and $L:\T^n\times\R^n\to\R$ is the corresponding Lagrangian, i.e., $L(x,q):=\sup_{p\in\R^n}\{p\cdot q-H(x,p)\}$.  
We here choose the control $\xi$ and also the \textit{stopping time} $\theta$ 
and therefore, this problem is called the \textit{optimal stopping} problem 
in the context of control problem. 
It is well-known that the \textit{value function} associated with this is the unique viscosity solution of (C). See \cite{B,BCD} for details. 
Let $u$ be the value function for the optimal stopping problem,  then  
\begin{align*}
u(x,t)\le&\, 
\E_{x}\Big[\int_{0}^{t}
L(X^{\xi}(s),\xi(s))\,ds
+u_0(X^{\xi}(t))
\Big] 
\end{align*}
for all admissible control $\xi$. 
Therefore, if we take the infimum on $\xi$ in the right hand side and denote it by $U$, then 
we get $u(x,t)\le U(x,t)$ on $\cQ$.
Notice that $U$ is the solution of the initial-value  
problem for the degenerate viscous Hamilton--Jacobi equations, whose
large-time behavior has been studied in \cite{CGMT}. 
By \cite[Theorem 1.1]{CGMT}, we have $U(x,t)+c_Ht$ converges a solution of 
the ergodic problem (E). 
Thus, if $c_H>0$, then we see $u(x,t)\to-\infty$ as $t\to \infty$. 
In particular, $u(x,t)<\psi(x)$ for all $(x,t)\in\T^n\times[T_0,\infty)$ with a sufficient large constant
$T_0>0$.  
\end{rem}

Therefore, we only need to consider the large-time behavior of solutions to (C) in the case where $c_H\le0$, which is studied in the next section.  
We give the proof of Proposition \ref{prop:EO} at the end of this section. 

\begin{proof}[Proof of Proposition {\rm\ref{prop:EO}}]
Take $v$ to be a solution of (E) and take $C>\|v\|_{L^\infty(\T^n)}+\|\psi\|_{L^\infty(\T^n)}+|c_H|$, then $v-C \leq \psi$. We have that $v-C$ and $\psi$ are a subsolution and a supersolution of (EO) respectively. By using the Perron method, we achieve the existence of a solution $V\in C(\T^n)$ of (EO).

Now we prove that $V$ is unique in case $c_H<0$. Take $V_1, V_2$ to be two solutions of (EO). Notice that
$$
\max\{-\tr\big(A(x)D^2(v-C)\big)+H(x,D(v-C)), v-C-\psi\} \leq c_H <0.
$$
For any $\lam\in (0,1)$, we use the convexity of $H$ in $p$ to deduce that 
$\tilde V_1=\lam V_1+(1-\lam)(v-C)$ is a subsolution of
$$
\max\{-\tr\big(A(x)D^2\tilde V_1 \big)+H(x,D\tilde V_1), \tilde V_1-\psi\} \leq (1-\lam)c_H <0.
$$
By the continuity of $\tilde V_1, V_2$ in $\T^n$, we achieve the existence of a small constant $\al>0$ so that
\begin{multline}\label{uni-c}
\al \tilde V_1+\max\{-\tr\big(A(x)D^2\tilde V_1 \big)+H(x,D\tilde V_1), \tilde V_1-\psi\}\\
\leq \al  V_2+\max\{-\tr\big(A(x)D^2 V_2 \big)+H(x,D V_2),  V_2-\psi\}
\end{multline}
The usual comparison principle yields that $\tilde V_1 \leq V_2$. Let $\lam \to 1$ to deduce further that $V_1 \leq V_2$. By the same argument, $V_2 \leq V_1$. 
The proof is complete.
\end{proof}

\section{Large-Time Behavior in case $c_H\leq 0$} 
In this section, we \textit{always} assume $c_H\le0$. 
In what follows, we follow the method which was introduced in \cite{CGMT}. 
In \cite{CGMT}, we take notice of the two new perspectives:  
\begin{enumerate}
\item
the conservation of energy, and
the long time averaging effects on the Hamiltonian and the diffusion
terms in the context of the nonlinear adjoint method,  
\item
estimates on the difference of gradient and Hessian of solutions of
the approximate equation and ergodic problem
on the support of the solution of the adjoint equation.
Let us emphasize that the estimates on the integral of the support of the 
solution of the adjoint equation are much better and stronger than the usual
known estimates on the whole torus. 
  
\end{enumerate}
As described in Introduction, if we consider the case $c_H\le0$, the obstacle 
influences the long time behavior of solutions, and therefore we need to 
consider the penalization terms for the approximations of (C) and (EO).

\subsection{Approximation} 
For each $\ep>0$, we introduce the rescaling function $u^\ep$ of $u$ as  $u^\ep(x,t)=u(x,t/\ep)$ for $(x,t)\in \Q$. It is clear that $u^\ep$ satisfies
$$
{\rm (C)_\ep} \quad
\max\{\ep u^\ep_t-\tr\big(A(x)D^2u^\ep\big)+H(x,Du^\ep), u^\ep-\psi\}=0 \quad \text{in} \ \Q.
$$

We introduce the approximation, for each $\del>0$,
\begin{equation}\notag
{\rm(A)_{\ep}^{\delta}}\quad
\begin{cases}
\ep w^{\ep, \delta}_t -\tr\big(A(x)D^2w^{\ep,\delta}\big)+ H(x,Dw^{\ep, \delta}) 
+\gam^{\del}(w^{\ep,\delta}-\psi)=\del^{2}\Del w^{\ep,\del}\ 
& \text{ in }\Q,  \\
w^{\ep, \del}(x,0)=u_{0}(x)  & \text{ on } \T^{n}
\end{cases}
\end{equation} 
 where $\gam^\del$ is the penalized function defined as $\gam^{\del}(r):=\gam(\del^{-1/4}r)$, where 
the function $\gam:\R\to[0,\infty)$ is defined by 
\[
\gam(r):=0 \quad \text{for} \ r\in(-\infty,0], \quad
\gam(r):=\frac{r^2}{2} \quad \text{for} \ r\in(0,\infty). 
\]
One could hence write that
$$
\gam^\del(r)=\frac{r_+^2}{2 \del^{1/2}},
$$
where $r_+=\max\{r,0\}$. 
Note that the scales we choose here are crucial in our analysis, 
which will be pointed out clearly later.

\begin{prop}\label{prop:est1} 
There exist a unique smooth solution $w^{\ep,\del}$ of ${\rm (A)}_\ep^\del$ 
and a universal constant $C>0$ independent of $\ep,\del$ such that 
\begin{itemize}
\item[(i)]
$\gam^{\del}\big( (w^{\ep,\del}-\psi)(x,t)\big)\le C \quad \text{for all} \ (x,t)\in\cQ$, 
\item[(ii)]
$\|w^{\ep,\del}\|_{\Li(\Q)}+\|Dw^{\ep,\del}\|_{\Li(\Q)}\le C$.
 \end{itemize}
\end{prop}

\begin{proof}



For $\delta<1$, we note that $\psi+C\delta^{1/4}$ is a supersolution of (A)$_\ep^\del$ provided that
$$
\frac{C^2}{2} \geq  \max_{x\in\T^n}\big\{
 |\tr\big(A(x)D^2\psi(x)\big)|+|H(x,D\psi(x))|+|\Del \psi(x)|
 \big\}.
$$
Note further that $u_0 \leq \psi \leq \psi+C\delta^{1/4}$. Thus, the comparison principle implies that $w^{\ep,\del} \leq \psi+C\delta^{1/4}$. So for $(x,t) \in \Q$,
$$
\gam^\del((w^{\ep,\del}-\psi)(x,t)) \leq \gam^\delta(C\delta^{1/4})=\frac{C^2}{2}.
$$

By another application of the comparison principle, we immediately get the estimate on $w^{\ep,\del}$ itself. The estimate on $Dw^{\ep,\del}$ is straightforward in light of the usual Bernstein method and (i), and is omitted here. We refer the readers to \cite{CGT2} for the detailed proof.
\end{proof}

\begin{prop}\label{prop:stability}
Let $u$ be the viscosity solution of {\rm(C)}, and
by abuse of notation, $w^{\ep}$ be the solution of {\rm(A)}$_{\ep}^{\ep^2}$,
i.e., $w^\ep:=w^{\ep,\ep^2}$.
There exists $C>0$ independent of $\ep$ such that
\[
\|w^{\ep}(\cdot,1)-u(\cdot,1/\ep)\|_{\Li(\T^n)}=
\|w^{\ep}(\cdot,1)-u^\ep(\cdot,1)\|_{\Li(\T^n)}\le C\ep^{1/2}. 
\]
\end{prop}

In order to prove Proposition \ref{prop:stability}, we introduce 
the adjoint equation for the linearized equation of (A)$_{\ep}^{\del}$: 
\begin{equation} \notag
{\rm (AJ)_\ep^{\del}} \qquad 
\begin{cases}
-\ep \sig^{\ep,\del}_t -\big(a^{ij}(x)\sig^{\ep,\del}\big)_{x_i x_j}
-\text{div}\big(D_p H(x,Dw^{\ep,\del}) \sig^{\ep,\del}\big) \\
\hspace*{4cm}+(\gam^{\del})'(w^{\ep,\del}-\psi)\sig^{\ep,\del} 
=\del^2\Del\sig^{\ep,\del} & \text{ in }\T^n \times (0,1),  \\
\sig^{\ep,\del}(x,1)=\del_{x_0}  & \text{ on } \T^{n},
\end{cases}
\end{equation} 
where  $\del_{x_0}$ is the Dirac delta measure at some point $x_0 \in \T^n$. 
Since we often hereafter use the linearized operator, we set it as
\[
\cL^{\ep,\del}[f]:=
\ep f_t-a_{ij}f_{x_ix_j}+D_pH(x,Dw^{\ep,\del})\cdot Df+
(\gam^{\del})'(w^{\ep,\del}-\psi)f-\del^2\Del f 
\]
for $f\in C^2(\T^n\times[0,1])$. 

\begin{prop}[Elementary Properties of $\sig^{\ep,\del}$]\label{prop:element}
We have $\sig^{\ep,\del} \ge 0$ on $\T^n\times[0,1)$ and
\begin{itemize}
\item[(i)]
$\displaystyle
\dfrac{d}{dt}\int_{\T^n} \sig^{\ep,\del}(x,t)\,dx\ge0$ 
for all $t \in [0,1]$, 
\item[(ii)]
$\displaystyle
\iint_{\T^n\times[0,1]}
(\gam^{\del})'(w^{\ep,\del}-\psi)\sig^{\ep,\del}(x,t)\,dx\,dt
=\ep\int_{\T^n}\sig^{\ep,\del}(x,1)-\sig^{\ep,\del}(x,0)\,dx\le\ep$,  
\item[(iii)]
$\displaystyle
\ep\iint_{\T^n\times[0,1]}\sig^{\ep,\del}(x,t)\,dx\,dt
+
\int_{0}^{1}\int_{t}^{1}\int_{\T^n}
(\gam^{\del})'(w^{\ep,\del}-\psi)\sig^{\ep,\del}(x,s)\,dx\,ds\,dt
=\ep$.  \\
In particular, 
\[
\iint_{\T^n\times[0,1]}\sig^{\ep,\del}(x,t)\,dx\,dt\le1. 
\]
\end{itemize}
\end{prop}
The proof of Proposition \ref{prop:element} is straightforward
by using the maximum principle and usual integration techniques.

\begin{lem}\label{lem:est2}
The following holds for some constant $C>0$ independent of $\ep,\del$:
\[
\iint_{\T^n\times[0,1]}
\left(a^{ij}(x)w^{\ep,\del}_{x_ix_k}w^{\ep,\del}_{x_jx_k}+
\del^{2}|D^{2}w^{\ep,\del}|^2\right) \sig^{\ep,\del}\,dx\,dt\le C. 
\]
\end{lem}
\begin{proof}
Let $w^{\ep, \del}$ be the solution of (A)$_{\ep}^{\del}$
and set $\varphi(x,t):=|Dw^{\ep, \del}|^2/2$.  
Then $\varphi$ satisfies 
\begin{multline*}
\ep\varphi_t
-a^{ij}(\varphi_{x_ix_j}-w_{x_ix_k}^{\ep,\del} w_{x_jx_k}^{\ep,\del} )
-a^{ij}_{x_k}w_{x_ix_j}^{\ep,\del} w_{x_k}^{\ep,\del}
+D_pH\cdot D\varphi\\
+D_xH\cdot Dw^{\ep,\del}
+(\gam^{\del})'D(w^{\ep,\del}-\psi)\cdot Dw^{\ep,\del}
= 
\del^2(\Del\varphi-|D^2 w^{\ep,\del}|^2).  
\end{multline*}
By Proposition \ref{prop:est1} (ii), 
\begin{equation}\label{bb-1}
\cL^{\ep,\del}[\varphi]
+a^{ij}w_{x_ix_k}^{\ep,\del} w_{x_jx_k}^{\ep,\del}
+\del^{2}|D^2 w^{\ep,\del}|^2
\le
a^{ij}_{x_k}w_{x_ix_j}^{\ep,\del} w_{x_k}^{\ep,\del}
+
C((\gam^{\del})'+1). 
\end{equation} 
Note that due to the trace inequality \eqref{SV-a2} and Proposition \ref{prop:est1} (ii), 
we have for some $c>0$ small enough
\begin{align*}
&a^{ij}_{x_k}w^{\ep,\del}_{x_ix_j}w^{\ep,\del}_{x_k}
=\tr(A_{x_k}D^2 w^{\ep,\del})w^{\ep,\del}_{x_k}
\leq c \big(\tr(A_{x_k}D^2 w^{\ep,\del})\big)^2
+\frac{1}{4c}|Dw^{\ep,\del}|^2\\ 
\leq&\, 
\frac{1}{2} \tr(D^2 w^{\ep,\del}A D^2 w^{\ep,\del})+C
=\frac{1}{2}a^{ij} w^{\ep,\del}_{x_i x_k} w^{\ep,\del}_{x_j x_k}+C.
\end{align*}

Multiplying \eqref{bb-1} by $\sig^{\ep,\del}$, using the above inequality, and integrating by parts on over $\T^n\times[0,1]$ to yield that
\begin{align*}
&\iint_{\T^n\times[0,1]} 
\left(a^{ij}w_{x_ix_k}^{\ep,\del} w_{x_jx_k}^{\ep,\del}
+\del^{2}|D^2 w^{\ep,\del}|^2\right)
\sig^{\ep,\del}\,dx\,dt\\
\le &\, 
\iint_{\T^n\times[0,1]} C((\gam^{\del})'(w^{\ep,\del}-\psi)+1)\sig^{\ep,\del}\,dx\,dt
\le C(\ep+1) \le C, 
\end{align*}
where we used Proposition \ref{prop:element} (ii) in the second last inequality.
\end{proof}

\begin{lem}\label{lem:est3}
We have 
\[
\Big\|\frac{\pl}{\pl\del}w^{\ep,\del}(\cdot,1)\Big\|_{\Li(\T^n)}\le \frac{C}{\ep}+\frac{C}{\del^{3/4}}. 
\]
\end{lem}
\begin{proof}
Note first that $w^{\ep,\del}$ is differentiable with respect to $\del$ by a standard regularity result for parabolic equations.   
Differentiating the equation in (A)$_{\ep}^{\del}$ with respect to 
$\del$, we get  
\begin{multline*}
\varepsilon (w^{\varepsilon,\del}_{\del})_t 
-a^{ij}(w_{\del}^{\varepsilon,\del})_{x_ix_j}+ D_p H(x,D w^{\varepsilon,\del}) \cdot D w^{\varepsilon,\del}_{\del}\\
+\gam'\Big(\frac{w^{\varepsilon,\del}-\psi}{\del^{1/4}}\Big)
\cdot\left(\frac{w^{\varepsilon,\del}_{\del}}{\del^{1/4}}
-\frac{w^{\varepsilon,\del}-\psi}{4\del^{5/4}}\right)
= \del^2\Delta w^{\varepsilon,\del}_{\del} 
+ 2\del\Delta w^{\varepsilon,\del} \ 
\text{ in } \T^n, 
\end{multline*}
where $f_{\del}$ denotes the derivative of the function $f$ with respect 
to the parameter $\del$. 
Note that
\begin{align*}
\gam'\Big(\frac{w^{\varepsilon,\del}-\psi}{\del^{1/4}}\Big)
\cdot\left(\frac{w^{\varepsilon,\del}_{\del}}{\del^{1/4}}
-\frac{w^{\varepsilon,\del}-\psi}{4\del^{5/4}}\right)
&=\, 
(\gam^{\del})'(w^{\varepsilon,\del}-\psi)
\cdot\left(w^{\varepsilon,\del}_{\del}
-\frac{w^{\varepsilon,\del}-\psi}{\del^{1/4}}\cdot\frac{1}{4\del^{3/4}}
\right)\\
&\ge\,
(\gam^{\del})'(w^{\varepsilon,\del}-\psi)
\cdot\left(w^{\varepsilon,\del}_{\del}-\frac{C}{\del^{3/4}}\right)
\end{align*}
in light of Proposition \ref{prop:est1} (i). 
Multiplying the above by $\sigma^{\varepsilon, \del}$, integrating by parts on $\T^n\times[0,1]$,
and noting that $w^{\ep,\del}_\del(\cdot,0)=0$, we get 
\begin{align*}
&|\varepsilon w^{\varepsilon,\del}_{\del} (x_0,1)| \\
\le&\, 
\iint_{\T^n\times[0,1]}2\del|\Delta w^{\varepsilon,\del} \sigma^{\varepsilon, \del}| \, dx \, dt
+\frac{C}{\del^{3/4}}\iint_{\T^n\times[0,1]}(\gam^{\del})'(w^{\varepsilon,\del}-\psi)\sig^{\ep,\del} \, dx \, dt\\
\le&\, 
C\left(\iint_{\T^n\times[0,1]}\del^2|D^2 w^{\varepsilon,\del}|^2 \sigma^{\varepsilon, \del}\, dx \, dt\right)^{1/2}\cdot\left( \iint_{\T^n \times [0,1]}\sig^{\ep,\del}\,dxdt\right)^{1/2}
+\frac{C\ep}{\del^{3/4}}\\
\le&\, C+\frac{C\ep}{\del^{3/4}}
\end{align*}
by Lemma \ref{lem:est2} and Proposition \ref{prop:element} (ii). 
By choosing properly the point $x_0$ we have thus
\[
\| w^{\varepsilon,\del}_{\del} (\cdot, 1) \|_{L^{\infty}(\T^n)} 
\leq \frac{C}{\varepsilon}+\frac{C}{\del^{3/4}}.
\qedhere\]
\end{proof}

Proposition \ref{prop:stability} is a straightforward result of Lemma 
\ref{lem:est3} with $\del=\ep^2$ as
\begin{multline*}
|w^\ep(x_0,1)-u^\ep(x_0,1)|=|w^{\ep,\ep^2}(x_0,1)-w^{\ep,0}(x_0,1)|\\
\leq \int_0^{\ep^2}| w^{\ep,\del}_\del(x_0,1)|\,d\del
=\frac{C \ep^2}{\ep}+C \left(\ep^2\right)^{1/4}=C\ep+C\ep^{1/2} \leq C\ep^{1/2}.
\end{multline*} 
We, henceforth, write $w^{\ep}$, $\gam^{\ep}$, $\cL^{\ep}$, (A)$_{\ep}$ 
and (AJ)$_{\ep}$ for $w^{\ep,\ep^2}$, $\gam^{\ep^2}$, $\cL^{\ep,\ep^2}$, 
(A)$_{\ep}^{\ep^2}$ and (AJ)$_{\ep}^{\ep^2}$, respectively, 
for the simplicity of notation.


\subsection{Approximated Ergodic Problems}
We next recall the result on the approximated ergodic problems
for degenerate viscous Hamilton--Jacobi equations without the obstacle term
(see \cite[Proposition 2.2]{CGMT} and \cite[Subsection 2.4]{CGMT}).
\begin{prop}\label{prop.AE}
For each $\ep \in (0,1)$, there exists a unique constant $c_H^\ep$ such that 
the approximated ergodic problem
$$
{\rm (E)_\ep}\quad -\tr\big(A(x)D^2 v^\ep\big)+H(x,Dv^\ep)=\ep^4\Del v^\ep + c_H^\ep \quad \text{in}\ \T^n
$$
has a unique solution $v^\ep\in C^2(\T^n)$ up to some additive constants. Moreover,
$$
|c_H^\ep-c_H|\le C\ep^2 \quad \text{and} \quad \|Dv^\ep\|_{L^\infty(\T^n)} \leq C
$$
for some positive constant $C$ independent of $\ep$.
\end{prop}

We can prove the above proposition by using a similar argument to the proof of Proposition \ref{prop:e.CGMT}.
We now study the approximated ergodic problem with the penalized terms. 
\begin{prop}\label{prop:AEO}
For each $\ep \in (0,1)$, set $c^\ep:=\max\{0,c_H^\ep\}$. 
Then the approximated equation to {\rm(EO)} 
$$
{\rm (EO)_\ep} \quad
-\tr\big(A(x)D^2V^{\ep}\big)+H(x,DV^\ep)+\gam^{\ep}(V^\ep-\psi)
=\ep^4\Del V^\ep +c^\ep \quad \textnormal{in} \ \T^n. 
$$
has a  solution $V^\ep \in C^2(\T^n)$.
Moreover, 
\[
0\le c^\ep\le C\ep^2 \quad \text{and} \quad \|DV^\ep\|_{L^\infty(\T^n)} \leq C.
\] 
\end{prop}

\begin{proof}
 Pick $v^\ep$ to be a solution of (E)$_\ep$.
One could see that $v^\ep-C$ and $\psi+C\ep^{1/2}$ are respectively a subsolution and a supersolution of (EO)$_\ep$ for $C>0$ sufficiently large. We then apply the Perron method to achieve the existence of a solution $V^\ep$ of (EO)$_\ep$. 
Noting that we are assuming $c_H\le0$, we easily see $c^\ep\to0$ as $\ep\to0$. 
\end{proof}


\subsection{Stability Result and Proof of Main Theorem}
As observed in \cite{CGMT},  Theorem \ref{thm:main} could be obtained easily
as a corollary of the following key stability result: 

\begin{thm}\label{thm:key}
We have
$$
\lim_{\ep \to 0}\ep \|w^\ep_t(\cdot,1)\|_{L^\infty(\T^n)} = 0.
$$
\end{thm}
We first recall the proof of Theorem \ref{thm:main} for self-containedness 
and postpone the proof of Theorem \ref{thm:key} to the next subsection as it involves many technical issues. 

\begin{proof}[Proof of Theorem {\rm\ref{thm:main}} in case {\rm(ii)}]
We use Proposition \ref{prop:est1} to yield the existence of a sequence $\{\ep_m\} \to 0$ so that $w^{\ep_m}(\cdot,1)$ converges uniformly to a function $V\in C(\T^n)$. In view of Theorem \ref{thm:key}, $V$ is a solution of (EO) and thus a time-independent solution of the equation in (C) and (C)$_\ep$. Set $t_m=\ep_m^{-1}$ and use Proposition \ref{prop:stability} to achieve that $u(\cdot,t_m)$ converges uniformly to $V$ in $\T^n$. 

We show that $u(\cdot,t)$ converges uniformly to $V$ as $t\to \infty$. 
For any $t>0$, we pick $m\in \N$ so that $t_m \leq t <t_{m+1}$ and
use the comparison principle to deduce 
$$
\|u(\cdot,t)-V\|_{L^\infty(\T^n)} =\|u(\cdot,t_m+(t-t_m))-V\|_{L^\infty(\T^n)}
\leq \|u(\cdot,t_m)-V\|_{L^\infty(\T^n)}.
$$
Let $m\to \infty$ in the above to yield the desired result.
\end{proof}


\subsection{Key Estimates}

The following three Lemmas provide the key ingredients to establish
Theorem \ref{thm:key}.  

\begin{lem}[Conservation of Energy] \label{lem-1}
The followings hold:
\begin{itemize}
\item[(i)]
$\displaystyle \dfrac{d}{dt} \int_{\T^n} 
\big[a^{ij}w^{\ep}_{x_ix_j}+\ep^4\Del w^\ep-H(x,Dw^\ep)-\gam^{\ep}(w^\ep-\psi)\big]\sig^\ep \,dx=0,$ \\
\item[(ii)]
$\displaystyle 
\ep w^\ep_t(x_0,1)=
\iint_{\T^n\times[0,1]} 
\big[a^{ij}w^{\ep}_{x_ix_j}+\ep^4\Del w^\ep
-H(x,Dw^\ep)-\gam^{\ep}(w^\ep-\psi)\big]\sig^\ep \,dx\,dt.$
\end{itemize}
\end{lem}

\begin{proof}
We observe that $w^\ep_t$ solves the linearized equation, 
i.e., $\cL^{\ep}[w^\ep_t]=0$. 
Multiply this by $\sig^\ep$ and integrate over $\T^n$ to deduce 
$$
\frac{d}{dt} \int_{\T^n} \ep w^\ep_t \sig^\ep\,dx=0,
$$
which is precisely (i). Integrate (i) over $[0,1]$ with respect to $t$ to achieve (ii).
\end{proof}

We hereafter set $W^{\ep}(x,t):=V^{\ep}(x)-w^{\ep}(x,t)$ for all $(x,t)\in\T^n\times[0,1]$, 
where $V^{\ep},w^\ep$ are a solution of (EO)$_\ep$ and (A)$_\ep$, respectively.   

\begin{lem}[Key Estimates 1] \label{lem:key1}
There exists a positive constant $C$, independent of $\ep$, such that
the followings hold{\rm:}
\begin{itemize}
\item[(i)]
$\displaystyle
\iint_{\T^n\times[0,1]} 
|DW^\ep|^2\sig^{\ep}\,dx\,dt
\le C\ep,$
\item[(ii)]
$\displaystyle
\iint_{\T^n\times[0,1]}
\big(\gam^{\ep}(V^\ep-\psi)+\gam^{\ep}(w^\ep-\psi)\big)\sig^{\ep}\,dxdt
\le C\ep.$ 
\end{itemize}
\end{lem}

\begin{lem}[Key Estimates 2] \label{lem:key2}
There exists a positive constant $C$, independent of $\ep$, such that
the followings hold{\rm:}
\begin{itemize}
\item[(i)]
$\displaystyle
\iint_{\T^n\times[0,1]} 
\ep^7|D^2W^\ep|^2 \sig^\ep\,dx\,dt
\le C,$
\item[(ii)]
$\displaystyle
\iint_{\T^n\times[0,1]}
a^{ij}(x) a^{ll}(x) W^\ep_{x_i x_k}  W^\ep_{x_j x_k}
\sig^{\ep}\,dxdt\le C\sqrt{\ep},$  
\item[(iii)]
$\displaystyle
\iint_{\T^n\times[0,1]}
\big|a^{ij}(x)W^\ep_{x_i x_j}\big|^2\sig^{\ep}\,dxdt\le C\sqrt{\ep}$.  
\end{itemize}
\end{lem}

Let us first present the proof of Theorem \ref{thm:key}
by using Lemmas \ref{lem-1}, \ref{lem:key1} and \ref{lem:key2}
before entering the technical computations of Lemmas \ref{lem:key1}, \ref{lem:key2}.
\begin{proof}[Proof of Theorem {\rm\ref{thm:key}}]
Choose $x_0$ such that
$\ep|w^\ep_t(x_0,1)|=\ep\|w^\ep_t(\cdot,1)\|_{\Li(\T^n)}$. 
Thanks to Lemma \ref{lem-1}, and the existence of solutions of (EO)$_\ep$, 
\begin{align*}
&\ep \| w^\ep_t (\cdot,1) \|_{L^{\infty}(\T^n)}=\ep|w^\ep_t(x_0,1)| \\
=&\, \left| \iint_{\T^n\times[0,1]} 
\big[a^{ij}w^{\ep}_{x_ix_j}-H(x,Dw^\ep)-\gam^{\ep}(w^{\ep}-\psi)
+\ep^4\Del w^\ep\big] \sig^\ep \,dx\,dt \right| \\
\leq&\, 
\Big| \iint_{\T^n\times[0,1]} 
\big[a^{ij}(w^{\ep}-V^{\ep})_{x_ix_j}-H(x,Dw^\ep)+H(x,DV^\ep)\\
&\hspace*{3cm}
-\gam^{\ep}(w^{\ep}-\psi)+\gam^{\ep}(V^{\ep}-\psi)
+\ep^4\Del (w^\ep-V^{\ep})\big] \sig^\ep \,dx\,dt \Big| + |c^\ep|. 
\end{align*}
By Proposition \ref{prop:element} (iii), Lemmas \ref{lem:key1} and \ref{lem:key2}, 
\begin{align*}
&\ep \| w^\ep_t (\cdot,1) \|_{L^{\infty}(\T^n)}\\
\le&\, 
C\iint_{\T^n\times[0,1]} 
\big[ 
|a^{ij}W^\ep_{x_ix_j}|+|DW^\ep| \\
&\hspace*{3cm}
+|\gam^{\ep}(w^{\ep}-\psi)-\gam^{\ep}(V^{\ep}-\psi)|
+\ep^4|\Del W^\ep| \big] \sig^\ep \,dx\,dt + C \ep^2 \\
\le&\, 
C\left(\iint_{\T^n\times[0,1]} 
\big|a^{ij}W^\ep_{x_ix_j}\big|^2\sig^{\ep}\,dx\,dt\right)^{1/2}\cdot
\left(\iint_{\T^n\times[0,1]} \sig^{\ep}\,dx\,dt\right)^{1/2}\\
&+
C\left(\iint_{\T^n\times[0,1]} 
|DW^\ep|^2\sig^{\ep}\,dx\,dt\right)^{1/2} \cdot
\left(\iint_{\T^n\times[0,1]} \sig^{\ep}\,dx\,dt\right)^{1/2}
\\
&
+
\iint_{\T^n\times[0,1]} 
\big(\gam^{\ep}(w^{\ep}-\psi)+\gam^{\ep}(V^{\ep}-\psi)\big)
\sig^{\ep}\,dx\,dt\\
&
+C\ep^4\left(\iint_{\T^n\times[0,1]} 
|D^2W^\ep|^2 \sig^\ep \,dx\,dt\right)^{1/2}\cdot
\left(\iint_{\T^n\times[0,1]} \sig^{\ep}\,dx\,dt\right)^{1/2}
 + C \ep^2 \\
\le&\, 
C \ep^{1/4}, 
\end{align*}
which finishes the proof. 
\end{proof}

\begin{proof}[Proof of Lemma {\rm\ref{lem:key1}}]
Notice first that 
\[
\gam(r)-\gam(s)\ge\gam'(s)(t-s)+\frac{1}{2}|r_{+}-s_{+}|^2 \ \text{for all} \ r,s\in\R, 
\]
where $r_{+}=\max\{0,r\}$. 
Thus, 
\begin{align*}
\gam^{\ep}(V^\ep-\psi)-\gam^{\ep}(w^\ep-\psi)
\ge&\, 
(\gam^{\ep})'(w^\ep-\psi)\cdot(V^\ep-w^\ep)
+\frac{1}{2}\Big|\frac{(V^{\ep}-\psi)_{+}-(w^{\ep}-\psi)_{+}}{\ep^{1/2}}\Big|^2.
\end{align*}

Subtracting equation (A)$_{\ep}$ from (EO)$_\ep$,
by the above inequality and the uniform convexity of $H$, we get
\begin{align*}
0=&\,
\ep (V^\ep-w^\ep)_t -a^{ij}(V^\ep-w^\ep)_{x_ix_j}
+ H(x,DV^\ep)-H(x,Dw^\ep)\\
&+\gam^{\ep}(V^{\ep}-\psi)-\gam^{\ep}(w^{\ep}-\psi)
-\ep^4\Del (V^\ep-w^\ep)-c^\ep\\
\ge&\, 
\cL^{\ep}[W^{\ep}]
+\theta | DW^\ep|^2
+\frac{1}{2}\Big|\frac{(w^{\ep}-\psi)_{+}-(V^{\ep}-\psi)_{+}}{\ep^{1/2}}\Big|^2
-c^\ep.
\end{align*}
Multiply the above inequality by $\sig^\ep$ and integrate by parts on $[0,1] \times \T^n$ to deduce that
\begin{equation} \label{key-1}
\iint_{\T^n\times[0,1]}\left( |DW^\ep|^2
+\Big|\frac{(w^{\ep}-\psi)_{+}-(V^{\ep}-\psi)_{+}}{\ep^{1/2}}\Big|^2\right) \sig^\ep\,dx\,dt \le C \ep,  
\end{equation}
implies (i). 
We next prove (ii).
By the explicit formula of $\gam^\ep$, 
\begin{align}
&\iint_{\T^n \times [0,1]} \gam^\ep(w^\ep-\psi)\sig^\ep\,dx\,dt
=\iint_{\T^n \times [0,1]}  \frac{ (w^\ep-\psi)^2_+}{2\ep} \sig^\ep\,dx\,dt \notag\\
=\,&
\iint_{\T^n \times [0,1]} \frac{(w^\ep-\psi)_+}{2} (\gam^\ep)'(w^\ep-\psi) \sig^\ep\,dx\,dt \notag\\
\leq\,&
 C\ep^{1/2} \iint_{\T^n \times [0,1]}  (\gam^\ep)'(w^\ep-\psi) \sig^\ep\,dx\,dt
\leq C \ep^{3/2}, \label{key-add}
\end{align}
where we used Proposition \ref{prop:est1} (i) and Proposition \ref{prop:element} (ii)
in the last two inequalities. 
Combining \eqref{key-1} and \eqref{key-add}, we deduce further that
\begin{align*}
&\iint_{\T^n \times [0,1]} \gam^\ep(V^\ep-\psi)\sig^\ep\,dx\,dt
=\iint_{\T^n \times [0,1]} \frac{(V^\ep-\psi)_+^2}{2\ep}\sig^\ep\,dx\,dt\\
\leq \,&
\iint_{\T^n \times [0,1]} \frac{(w^\ep-\psi)_+^2}{\ep}\sig^\ep\,dx\,dt
+
\iint_{\T^n\times[0,1]}
\Big|\frac{(w^{\ep}-\psi)_{+}-(V^{\ep}-\psi)_{+}}{\ep^{1/2}}\Big|^2 \sig^\ep\,dx\,dt\\
\leq\, &
C(\ep^{3/2}+\ep)\le C\ep,
\end{align*}
which implies (ii). 
\end{proof}

\begin{proof}[Proof of Lemma {\rm\ref{lem:key2}}]
Subtract (A)$_{\ep}$ from (EO)$_\ep$ and differentiate with respect to $x_k$ to get
\begin{multline*}
\ep (V^\ep-w^\ep)_{x_k t} -a^{ij} ( (V^\ep-w^\ep)_{x_k})_{x_i x_j} 
-a^{ij}_{x_k}(V^\ep-w^\ep)_{x_i x_j}\\
+D_p H(x,DV^\ep)\cdot DV^\ep_{x_k} - D_p H(x,Dw^\ep) \cdot Dw^\ep_{x_k}
+ H_{x_k}(x,DV^\ep)-H_{x_k}(x,Dw^\ep)\\
+(\gam^{\ep})'(V^\ep-\psi)(V^\ep_{x_k}-\psi_{x_k})
-(\gam^{\ep})'(w^\ep-\psi)(w^\ep_{x_k}-\psi_{x_k})
= \ep^4 \Del(V^\ep-w^\ep)_{x_k}.
\end{multline*}
Let $\varphi(x,t):=|DW^\ep|^2/2$. 
Multiplying the last identity by $W^\ep_{x_k}$
and summing up with respect to $k$, we achieve that
\begin{align}
&L^\ep[\varphi]+(\gam^{\ep})' \varphi+\ep^4|D^2W^\ep|^2\nonumber\\
=&\,
-a^{ij}W^\ep_{x_ix_k}W^\ep_{x_jx_k}
+a^{ij}_{x_k}W^\ep_{x_ix_j}W^\ep_{x_k}\nonumber\\
&
-\left[ \Big( D_p H(x,DV^\ep)-D_p H(x,Dw^\ep) \Big) \cdot DV^\ep_{x_k} \right]
W^\ep_{x_k}\nonumber\\
&-\Big( D_x H(x,DV^\ep)-D_x H(x,Dw^\ep) \Big) \cdot DW^\ep
\nonumber\\
&
-\big((\gam^{\ep})'(V^\ep-\psi)-(\gam^{\ep})'(w^\ep-\psi)\big)
D(V^{\ep}-\psi)\cdot DW^\ep. \label{eq:subtract}
\end{align}
We will use this identity to bound the integral of 
$\ep^4|D^2W^\ep|^2$ on the support $\sig^\ep$ on $\T^n\times[0,1]$
by estimating the terms on the right side. Note first that $(\gam^\ep)' \varphi \geq 0$.

By the same computation as in \eqref{Bern-1}, we have 
\begin{align}
&\iint_{\T^n\times[0,1]}
\left(-a^{ij}W^{\ep}_{x_ix_k}W^{\ep}_{x_jx_k}
+a^{ij}_{x_k}W^{\ep}_{x_ix_j}W^{\ep}_{x_k}\right)\sig^{\ep}\,dx\,dt \notag\\
\le&\, 
C\iint_{\T^n\times[0,1]}
|DW^\ep|^2\sig^{\ep}\,dx\,dt
\le C\ep. \label{comb-1}
\end{align}
Also, we have 
\begin{align}
&
\iint_{\T^n\times[0,1]}
\left|\left[ \Big( D_p H(x,DV^\ep)-D_p H(x,Dw^\ep) \Big) \cdot DV^\ep_{x_k} \right]
W^\ep_{x_k}\right|\sig^{\ep}\,dx\,dt\notag\\
\le&\, 
C \iint_{\T^n\times[0,1]}
|D^2 V^\ep|\  |DW^\ep|^2\sig^{\ep}\,dx\,dt\notag\\
\le&\, 
C \iint_{\T^n\times[0,1]}
\big[|D^2 W^\ep|\  |DW^\ep|^2 + |D^2 w^\ep|\  |DW^\ep|^2\big]
\sig^{\ep}\,dx\,dt\notag
\\
\le&\, 
\iint_{\T^n\times[0,1]}
\Big[\dfrac{\ep^4}{2}  |D^2W^\ep|^2
+\dfrac{C}{\ep^4} |DW^\ep|^2\Big]\sig^{\ep}\,dx\,dt \notag\\
&+ C
\left(\iint_{\T^n\times[0,1]}
|D^2 w^\ep|^2\sig^{\ep}\,dx\,dt\right)^{1/2}\cdot
\left(\iint_{\T^n\times[0,1]}\sig^{\ep}\,dx\,dt\right)^{1/2}\notag\\
\le&\, 
\dfrac{\ep^4}{2}\iint_{\T^n\times[0,1]}
 |D^2W^\ep|^2\sig^{\ep}\,dx\,dt
 +\frac{C}{\ep^4}\cdot\ep+\frac{C}{\ep^2}. \label{comb-2}
\end{align}
in view of Lemma \ref{lem:est2}.

Furthermore,
\begin{align}
&
\iint_{\T^n\times[0,1]}
\big|\big((\gam^{\ep})'(V^\ep-\psi)-(\gam^{\ep})'(w^\ep-\psi)\big)(V^{\ep}-\psi)_{x_k}W^{\ep}_{x_k}\big|\sig^{\ep}\,dx\,dt \notag\\
\leq&\,
C\iint_{\T^n\times[0,1]}
\frac{1}{\ep^{1/2}}\Big|\frac{(V^{\ep}-\psi)_{+}-(w^{\ep}-\psi)_{+}}{\ep^{1/2}}\Big|\cdot|DW^{\ep}|\sig^{\ep}\,dx\,dt\notag\\
\le&\,
\frac{C}{\ep^{1/2}}
\left(\iint_{\T^n\times[0,1]}
\Big|\frac{(V^{\ep}-\psi)_{+}-(w^{\ep}-\psi)_{+}}{\ep^{1/2}}\Big|^2\sig^{\ep}\,dx\,dt\right)^{1/2}
\cdot\left(\iint_{\T^n\times[0,1]}
|DW^{\ep}|^2\sig^{\ep}\,dx\,dt\right)^{1/2}\notag\\
\le&\,
C\ep^{1/2} \label{comb-3}
\end{align}
in light of \eqref{key-1} and Lemma \ref{lem:key1} (i). 
We used the scaling of the penalization crucially in the above computation. Indeed,
the behavior of $(\gam^\ep)'$ is pretty bad 
and the dangerous factor $\ep^{-1/2}$ appears in lines 2--3 of \eqref{comb-3}.
It is however being absorbed by the two good terms in line 3 of \eqref{comb-3}.

Combine \eqref{eq:subtract}--\eqref{comb-3} to get
\begin{align*}
\ep^4 \iint_{\T^n\times[0,1]} |D^2W^\ep|^2 \sig^\ep\,dx\,dt 
\le  \dfrac{C}{\ep^3},
\end{align*}
which is the conclusion of (i). 
\smallskip

We next prove (ii). 
We multiply \eqref{eq:subtract} by $a^{ll}$ and set 
$\Phi(x,t):=a^{ll}\varphi(x,t)$. By usual computations,
\begin{align*}
&L^\ep[\Phi]+(\gam^\ep)'(w^\ep-\psi)\Phi+\ep^4a^{ll}|D^2W^\ep|^2+a^{ll}a^{ij}W^{\ep}_{x_ix_k}W^{\ep}_{x_jx_k}\\
\le&\,
a^{ll}a^{ij}_{x_k}W^{\ep}_{x_ix_j}W^{\ep}_{x_k}-4a^{ll}_{x_i}a^{ij}W^{\ep}_{x_jx_k}W^{\ep}_{x_k}-2\ep^4Da^{ll}\cdot D\varphi
+C|DW^{\ep}|^2\\
&+a^{ll} |D^2V^\ep|\cdot |DW^\ep|+
a^{ll}
\left|(\gam^{\ep})'(V^\ep-\psi)-(\gam^{\ep})'(w^\ep-\psi)\big)
D(V^{\ep}-\psi)\cdot DW^{\ep}\right|,  
\end{align*}
where $p^{ij}, d^{i}$ are components of $P, D$ appearing  
in the proof of Proposition \ref{prop:e.CGMT}.

Notice that 
\begin{align*}
a^{ij}_{x_k} a^{ll} W^{\ep}_{x_i x_j} W^{\ep}_{x_k}
=\tr(A) \tr(A_{x_k} D^2W^{\ep})W^{\ep}_{x_k}
\leq
\frac{1}{4}a^{ll}a^{ij}W^{\ep}_{x_ix_k}W^{\ep}_{x_jx_k}+ C|DW^\ep|^2
\end{align*}
and, 
\begin{align*}
&-4a^{ll}_{x_i}a^{ij}W^{\ep}_{x_j x_k} W^{\ep}_{x_k}
=-4a^{ll}_{x_i} p^{mi} p^{mj} d^m W^{\ep}_{x_j x_k} W^{\ep}_{x_k}\\  
\leq&\,
C \left(\sum_l \sqrt{d^l}\right)\sum_{k,m} d^m \left|  p^{mj} W^{\ep}_{x_j x_k}\right| |DW^{\ep}|\\
\leq &\,
\frac{1}{4}\left( \sum_{l} d^l \right) \sum_{k,m} \left ( \sum_j \sqrt{d^m} p^{mj}W^{\ep}_{x_j x_k} \right)^2 +  C|DW^{\ep}|^2\\
=&\,
\frac{1}{4}a^{ll}a^{ij}W^{\ep}_{x_ix_k}W^{\ep}_{x_jx_k}+ C|DW^{\ep}|^2. 
\end{align*}
Noting that 
$(\gam^\ep)'(w^\ep-\psi)\Phi+\ep^4a^{ll}|D^2W^\ep|^2\ge0$, 
by the above computations with 
Lemma \ref{lem:key1} (i), \eqref{comb-3}, 
we get the conclusion of (ii).

We finally show that (iii) is a direct consequence of (ii). Indeed,
\begin{align*}
&\left(a^{ij}W^\ep_{x_ix_j}\right)^2
=\left(p^{mi} p^{mj} d^m W^\ep_{x_i x_j} \right)^2\\
\leq &\,
C\left(\sum_{j,m}\left|\sum_{i}p^{mi}d^m W^\ep_{x_ix_j}\right| \right)^2
=C\left(\sum_{j,m}\sqrt{d^m}\left|\sum_{i}p^{mi}\sqrt{d^m} W^\ep_{x_ix_j}\right| \right)^2\\
\leq &\,
C \sum_{j,m} d^m \left|\sum_{i}p^{mi}\sqrt{d^m} W^\ep_{x_ix_j} \right|^2
\leq C \left(\sum_l d^l \right)  \sum_{j,m}  \left|\sum_{i}p^{mi}\sqrt{d^m} W^\ep_{x_ix_j} \right|^2\\
=&\,
Ca^{ij} a^{ll} W^\ep_{x_i x_k} W^\ep_{x_j x_k}.
\end{align*}
The proof is complete.
\end{proof}


\end{document}